\documentclass[smallextended]{svjour3}       
\smartqed  
\usepackage{graphicx}
\usepackage{etoolbox,lastpage}
\usepackage{amsmath,amscd,amsfonts,amssymb,enumerate,tikz,fancyhdr}
\usepackage{color}
\usepackage[colorlinks]{hyperref}
\usepackage{geometry}
\usepackage{subfig, diagbox} 
\begin{document}

\title{Graded Galois Lattices and Closed Itemsets
\thanks{This article has been derived from the Ph. D. thesis by the first author under the supervision of the second author at \href{http://www.yu.ac.ir/index.aspx?siteid=30}{Yasouj University}, Yasouj, Iran.}
}

\titlerunning{Graded Galois Lattices and Closed Itemsets}        

\author{Reza Sotoudeh   \and  Hamidreza Goudarzi  \and  Ali Akbar Nikoukar}

\authorrunning{R. Sotoudeh  \and  H. R. Goudarz \and A. A. Nikoukar} 

\institute{R. Sotoudeh \at
              Department of Mathematics, Yasouj University, Yasouj, Iran. \\
              \email{reza.sotoudeh@yahoo.com}           
           \and
           H. R. Goudarz (Corresponding author) \at
              Department of Mathematics, Yasouj University, Yasouj, Iran. \\
              \email{goudarzi@yu.ac.ir} 
           \and
           A. A. Nikoukar \at
           Department of Mathematics, Yasouj University, Yasouj, Iran. \\
              \email{nikoukar@yu.ac.ir} 
}

\date{}

\maketitle

\begin{abstract}
The Galois lattice is a graphic method of representing knowledge structures. The first basic purpose in this paper is to introduce a new class of Galois lattices, called graded Galois lattices. As a direct result, one can obtain the notion of graded closed itemsets (sets of items), to extend the definition of 
closed itemsets. Our second important goal in this paper, is related to set a constructive method, computing the graded formal concepts and graded closed itemsets. We mean by a constructive method, a method that builds up a complete solution from scratch by sequentially adding components to a partial solution until the solution is complete. Besides of computational aspects, our methods in this paper are based on the strong results obtained by special mappings in the realm of domain theory. To reach the fertilized consequences and constructive algorithms, we need to push the study to the structures of Banach lattices.

\end{abstract}
\section{Introduction} \label{sec1}
Formal concept analysis (FCA) is an applicable reconstructing of lattice theory, although it has a deep root in the philosophy, its mathematical background  dates back to the works of Wille and Ganter (\cite{Ganter1999}, \cite{Wille1982}). The basic framework of  FCA starts with the data analysis of a given matrix form its data base which rows correspond to objects,  columns play as a set of attributes, and the values denote their relationship. 

In recent years, the comprehensive surveys show that, there is a good trend to apply the FCA methods in different areas. For example, as see in
\cite{Poelmans2013a}, \cite{Poelmans2013b} and \cite{Poelmans2014}, a total of 544 research papers have been collected from the prominent indexing system such as the Scopus, Google Scholars, leading scientific data bases such as ACM Digital Library, IEEE Xplore, Science Direct, Springer Links, etc. Also one can refer to some outstanding FCA conferences like ICFCA, ICCS, CLA, etc. \cite{Singh2016}. Among them, nearly 352 works have been based on their innovative contents. As FCA is concerned with the formalization of concepts and conceptual thinking, it has been applied in many disciplines such as software engineering \cite{Poshyvanyk2012}, machine learning (\cite{Ignatov2015b}, \cite{Korei2013} and \cite{Singh2017}), knowledge discovery (\cite{Aswani2011a}, \cite{Aswani2011b} and \cite{Aswani2014}) and ontology construction \cite{Jindal2020}, during the last 20-25 years.
\par
In details, one can remind some aspects of FCA in data mining and machine learning (see \cite{Ignatov2015}, \cite{Janostik2020}, \cite{Sumangali2017} and references therein):
\begin{itemize}
\item[•] 	The earlier works in FCA, behind the frequent itemset mining and association rules notions.
\item[•] 	Multimodal clustering (biclustering and triclustering).
\item[•] 	The role of FCA in classification: JSM-method, version spaces, and decision trees.
\item[•] 	Pattern structures for data with complex descriptions.
\item[•] 	FCA-based Boolean matrix factorization.
\item[•] 	Educational data mining cases study.
\item[•] 	Exercises with JSM-method in QuDA (Qualitative Data Analysis): solving classification task.
\end{itemize}

From another point of view, we may consider it as a mathematical technique. So, reinforcement and extension of notions seeing as algebraic or analytic objects make it as a powerful and enriched tool in applied areas. Over the past 30 years, the initial theory of FCA has been combined with and fertilized by other research domains in mathematics, including deception logics and conceptual graphs \cite{Singh2016}. To handle the uncertainty and vagueness in data, FCA has been successfully extended with a fuzzy setting, an interval--valued fuzzy setting, possibility theory, a rough setting and triadic concept analysis (\cite{Brito2018}, \cite{Dias2015}, \cite{Doerfel2012}, \cite{Poelmans2013b}, \cite{Poelmans2014}, \cite{Singh2017} and \cite{Yan2015}).

During the past recent years, fixed point theory has been a pioneer and prominent method to solve many problems in the territory of nonlinear analysis. Among the abundant results and theoretical consequences, one may consider a useful theorem related to Banach contraction theorem \cite{Agrawal2018}, specially, for its compatibility with discrete and ordered structures. Also its proof yields a constructive method for programmers to follow the new technical methods in their works. FCA methods have also benefited from its consequences extensively. In a complete lattice extracted from a transactional database, its formal concepts (and closed itemsets) are exactly correspondent to the fixed points of its closure operator \cite{Caspard2003}, and the way to compute them is based on iteration method. We believe that new theoretical and technical results from fixed point theory have extensively changed the face of FCA in the near future. So our main goal in this paper helps to realize this ambition.

The general idea and mathematical base of FCA have been formally presented by a tabular form of data. Also the key point to compute the Galois lattice is using the Galois connection. This can be realized by composition of Galois connectors which ends to the closure operator. Moreover, the existence of Galois lattice as a complete lattice, is undertook the main theorem of concept lattice \cite{Wille1982}. The first step in this paper is to enhance the data representation, from tabular contexts to multivalued functions. This generalization can help the researchers to benefit a rich source called fixed point theory. For example as the first measure, by applying the Banach contraction theorem on domains, one can compute the formal concepts and closed itemsets directly (without using the Galois connections). This can be done by iteration methods.

But the main purpose has been concentrated on setting up the new idea of graded concept lattices. While the structure of concept lattice has a hierarchical relation but grading the different levels of concepts as classified spectrums are so vital. For example as a direct result, this helps us to extend the closed labeled itemsets as the graded closed itemsets (see \cite{Garriga2008}). Moreover, the theoretical aspect behind these notions can be supported by approximate fixed point results which has a main role in applicable areas \cite{Alfuraidan2016}.

The ordered balls induced by this topology, on a Banach lattice, increase our ability to rank the dispersion of data. This new topological view point has many benefits even in applied territory. As a typical example, this yields another perception to rank the different sources pertained to a set of keywords in search engines.
\section{Preliminaries} \label{sec2}
Domain theory is a field of mathematics deals with some class of ordered sets, the ordered sets employing superma. The core idea of introducing domains can be summed up in this fact that they can divide the space with the help of ordered balls (their induced topology). Given that the space is also equipped with a vector space structure, all of these balls can be considered at the origin and with different radius. Therefore, considering that computer processes are generally a series of processes starting from a specific point, their compatibility with this type of mathematical structure become clearer  (\cite{Abramsky1994}, \cite{Keimel2017} and \cite{Stoltenberg2008}). So, let us mention an introduction to domain theory and other necessary prerequisites. 

\begin{definition} [\cite{Davey2012}] \label{def2.1}
Suppose that $(P,\leq)$ is a partially ordered set (poset). An element $m^* \in P$ is called the maximal element of $P$,
if for all $x \in P$ such that $m^* \leq x$, implies that $m^*=x$. The dual case for minimal element $m_*$ can be defined by the same way. For any subset $A$ of $P$, the greatest element of $A$ is an element $g \in A$ such that $x\leq g$, for all $x \in A$. The top element of $P$ is the greatest element of $P$. An upper bound of $A$ is an element $u \in P$ such that $ x \leq u$, for each $x \in A$. The supremum (infimum) of $A$ is the unique element $s \in P$ which is the minimal (maximal) element of the set of upper (lower) bounds of $P$. If the supremum (infimum) of $A$ exists, we will denote it by $\vee A$ ($\wedge A$). $A$ is said to be a directed set if each pair of elements in $A$ has supremum. $(P,\leq)$ is called a directed complete partially ordered set (dcpo) or a domain, if every directed subset of $P$ has supremum.
\end{definition}

\begin{definition} [\cite{Davey2012}] \label{def2.2}
Suppose that $(P, \leq)$ is a nonempty poset. If for any $x, y \in P$,
$$\inf \{x, y\} = x \wedge y \, , \sup \{x, y\}=x \vee y$$
exist, then
$(P, \leq)$ is called a lattice. If for any subset $A$ of $P$, $\wedge A$ and $\vee A$ exist, then we say that $(P, \leq)$ is a complete lattice.
\end{definition}

\begin{definition} [\cite{Schaefer1974}] \label{def2.3}
An ordered vector space is a real vector space $ P $ which is also an ordered space with the linear and order structures connected by the following implications:
\begin{itemize}
\item[(i)]
If $ x, y, z \in P $ and $x \leq y $, then $ x + z \leq y + z $,
\item[(ii)]
if $  x, y  \in P  $,  $  x \leq  y   $ and $  0  \leq  \alpha  \in  \mathbb{R}  $, then $  \alpha  x  \leq  \alpha  y. $
\end{itemize}
An ordered vector space, which is also a lattice, is a vector lattice or Riesz space. For a vector lattice $ P $ and $ x \in P $,
the positive part of $ x $ is $ x^{+} = x \vee 0 $ and negative part of $ x $ is $ x^{-} = (-x) \vee 0 $. 
The modulus of $ x $ is  defined as $ \mid x \mid = x \vee (-x) $.
\end{definition}

\begin{definition} [\cite{Schaefer1974}] \label{def2.4}
A normed lattice is a normed space which is also a vector lattice, in which
 $ \vert x \vert < \vert y \vert $ implies that $ \Vert x \Vert \leq \Vert y \Vert $.
A normed lattice which is also a Banach space is called a Banach lattice.
For each Banach lattice $(X, \|.\|)$, closed balls are defined by 
\begin{center}
$\overline{B}_r (x) = \{ y \in X : \|x-y\| \leq r \},$
\end{center}
where $x \in X$ and $r>0$. Let $[\overline{B}_r (x)]$ denotes the set of all closed balls on $X$ centered at $x$. For each $r,s >0$ and
 $x,y \in X$, one can define
 \begin{center}
$\overline{B}_r (x) \leq \overline{B}_s (y),$ iff $r \geq_{\mathbb{R}} s,$
\end{center}
where $r \geq_{\mathbb{R}} s $ denotes the ordinary order on real numbers.
\end{definition}

According to transitional property (induced by addition) in a Banach lattice, one can consider each ball $\overline{B}_r (x)$ as a closed ball centered at origin $0_X$. Since $[\overline{B}_r (x)]_{x \in X , r >0}$ can actually be considered as the set of all closed balls at origin, we denote  it by 
$[B_r (0_X)]_{ r > 0}$. Also this order can be considered as the reverse inclusion $\subseteq$. It is clear that 
$([B_r (0_X)]_{ r > 0} , \subseteq)$ is a poset (a nested sequence of closed balls). In the sequel, we assume that $(X , \|.\|)$
is a separable Banach lattice. We refer $([B_r (0_X)]_{ r > 0} , \subseteq)$ as $[B_r]$ and clearly  it is a domain 
Banach lattice. A sequence in $[B_r]$ will be denoted by $\{[x_n,r_n]\}$ \cite{Abramsky1994}.

\begin{definition} [\cite{Abramsky1994}] \label{def2.5}
Suppose $\{[x_n,r_n]\}$ is  an ascending sequence. We say that it is a Cauchy sequence, if 
$$(x_n - x_m) \xrightarrow{\parallel . \parallel} 0,$$
 while $(r_n - r_m) >0$. Also $\{[x_n,r_n]\}$ converges to $[x,r]$, if
 $x_n \xrightarrow{\parallel . \parallel} x$ and $r_n \longrightarrow r$. Especially, if $r_n \longrightarrow 0$, this means that $\{[x_n,r_n]\}$ 
 converges to $[x,0]$. So each maximal element can be contained in $[0_X,r]$, for some $r > 0$.
 \end{definition}
 
 \begin{theorem} [\cite{Abramsky1994}] \label{th2.6}
 Let $[B_r]$ be a domain, $\{[x_n,r_n]\}$ an ascending sequence and $[x,r]$ be an element of $[B_r]$. Then the following statements are equivalent:
 \begin{itemize}
\item[$\bullet$]
 $[x,r]$  is a least upper bound of $\{[x_n,r_n]\}$,
\item[$\bullet$]
 $[x,r]$  is an upper bound of $\{[x_n,r_n]\}$ and $r_n \longrightarrow r$,
 \item[$\bullet$]
 $x_n \xrightarrow{\parallel . \parallel} x$ and $r_n \longrightarrow r$.
 \end{itemize} 
 \end{theorem}
  
 \begin{definition} [\cite{Abramsky1994}] \label{def2.7}
 Suppose that $P_1$ and $P_2$ are two posets. A map $T : P_1 \longrightarrow P_2$ is called monotone increasing, if for all 
 $x_1 ,x^{'} _1 \in P_1$ such that $x_1 \leq x^{'}_1$ then $T(x_1) \leq T(x^{'}_1)$ in $P_2$. If $P_1$ and $P_2$ 
 are two dcpo's and $A$ is any directed subset of $P_1$, then we say that $T$ is Scott--continuous, provided $T(\vee A)=\vee T(A)$.
 For any Banach lattice 
 $(X,\|.\|)$ and a map $T : X \longrightarrow X$, we say that $T$ is sequentially convergent if for any $\{x_n\}$ such that 
 $\{Tx_n\}$ is convergent then $\{x_n\}$ is convergent. Moreover if $\{x_n\}$ has a convergent subsequence provided that 
 $\{Tx_n\}$ is convergent, $T$ is called subsequentially convergent.
 \end{definition}
\section{Main Results} \label{sec3}
The main purpose of this section is to provide the necessary preparations to prove a fundamental theorem. This will be the theoretical basis of our work in the next section. Usually, one of the most important conditions that leads to convergence is contraction. Different forms of contraction can be observed in metric and normed spaces, under the continuity and differentiability (for examples, see \cite{Hefti2015}, \cite{Ilic2011}, \cite{VanAn2015} and references therein). But when we study the discrete structures without a topological shape, the situation will be more difficult and therefore the need for special delicacy seems necessary. In addition, we must resort to constructive proofs, because providing such methods will enable us to benefit from their results in areas such as data processing. Theorem \ref{th3.4} will lead us to all these goals, but first we need to define an applicable type of contraction that will be used in the practical results of the next section.

\begin{definition}\label{def3.1}
Let $(X,\|.\|)$ be a Banach lattice and $f,g : X \longrightarrow X$ be any two maps. The map $g$ is said to be $f-$contraction, if there exists $k \in (0,1)$ 
such that  
\begin{center}
$\|f(g(x)) - f(g(y)) \| \leq k \|f(x) - f(y) \|,$
\end{center}
for all $x,y \in X$. If $f$ is the Identity map, then $g$ is a contraction. If $f$ is the Identity function,
then a Lipchitz constant for $g$ is a number $k \in \mathbb{R}^{+}$ such that for all $x,y \in X,$ 
\begin{center}
$\|g(x)- g(y)\| \leq k \| x-y \|.$
\end{center}
\end{definition}

Before presenting the main results, we need the following lemma.

\begin{lemma}\label{lem3.2}
$[B_r]$ is a dcpo, iff each sequence of $[B_r]$ converges to its maximal element.
\end{lemma}
\begin{proof}
Suppose that $\{ [x_n,r_n]\}$ is an ascending sequence in  $[B_r]$, then  $r_n \longrightarrow 0$ as $n \longrightarrow \infty$. 
So we must show that $[x_n,r_n] \longrightarrow [x,0]$, iff $[x,0]$ is a least upper bound of $[x_n,r_n].$
If $[x,0]$ is the limit point of $[x_n,r_n]$ then $\|x_n - x \| \longrightarrow 0$
and hence $x_n \xrightarrow{\parallel . \parallel} x, \; r_n \longrightarrow 0. $
So by Theorem \ref{th2.6}, $[x,0]$ is a least upper bound for $\{ [x_n,r_n]\}.$\\
Conversely, let $[x,0]$ is a least upper bound for $\{ [x_n,r_n]\}$. By Theorem \ref{th2.6}, $ x_n \xrightarrow{\parallel . \parallel} x$ and also 
$$(x_n - x) \xrightarrow{\parallel . \parallel} 0, \; r_n \longrightarrow 0.$$
But this means that $\lim_{n \rightarrow \infty} [x_n,r_n] = [r,0].$ \qed
\end{proof}

According to Definition  \ref{def3.1}, consider $T : [B_r] \longrightarrow [B_r]$ by $T([x,r])=[f(x),kr]$. Gaining by 
transitional property, we may suppose that  $[x,r_1] \subseteq [y,r_2]$, if $r_1 \leq r_2$ for all $x,y \in X, r_1 , r_2 \geq 0$. Now we have the following result.

\begin{theorem}\label{th3.3}
$T$ is monotone increasing and Scott--continuous.
\end{theorem}
\begin{proof}
Let $x,y \in X$ and $[x,r_1] \subseteq [y,r_2]$. Since $f$ is a contraction,
$$\|f(x)-f(y) \| \leq k \|x-y \|.$$
So
\begin{center}
$\|f(x)-f(y) \| \leq k \|x-y \| \leq kr_2 \leq kr_1.$
\end{center}
This means that $T([x,r_1]) \subseteq T([y,r_2])$ and hence $T$ is monotone. Now suppose that $[x_n,r_n]$ be an ascending sequence in $[B_r]$, with least upper bound $[x,0]$. Since $[B_r]$  is a domain and  $f$ is a contraction, then $T$ is monotone by the first part. So
$\{T[x_n,r_n]\}$ is also ascending. But $[x,0]$ is a least upper bound of  $[x_n,r_n]$,  then by Theorem \ref{th2.6}, 
$x_n \longrightarrow x$ and $r_n \longrightarrow 0$. Since  $f$ is continuous (indeed $f$ is uniformly continuous), we have 
$f(x_n) \xrightarrow{\parallel . \parallel} f(x), \; r_n \longrightarrow 0.$
Repeatedly by Theorem \ref{th2.6}, we imply that $[f(x),0]$ is the least upper bound of ascending sequence $[f(x_n),k^n r_n]$. Hence by definition,
$T(\vee [x_n,r_n]) = \vee T([x_n,r_n])$ and so $T$ is Scott--continuous.  \qed
\end{proof}

Now, we are ready to prove the main theorem in this section. But first we say that a map $T: X \longrightarrow X$ has a fixed point like 
$[x,r]$, if $f([x,r])=[x,r]$ (see \cite{Larsen2007}).

\begin{theorem}\label{th3.4}
Suppose that  $(X, \|.\|)$ is a Banach lattice and $f: X \longrightarrow X$ a contraction map with constant $k$.
Then,  $T: [B_r] \longrightarrow [B_r]$ defined by  $T([x,r])=[f(x),kr]$ has a unique fixed point. Moreover, the iterations $\{f^n (x)\}$
converges to the fixed point of  $T$, for any given $x \in X$.
\end{theorem}
\begin{proof}
For $x_0 \in X$,  we define the iterations $[f^n (x_0)]$,  where 
 \begin{center}
$x_{n+1}= f^n (x_n), \; n=0,1,2,...$\,.
\end{center}
Suppose that $[x_n,r_n]$ is any ascending sequence in $[B_r]$. According to Theorem  \ref{th3.3},  $T$ is monotone, i.e.,
\begin{center}
$[f(x_{n-1}),k^{n-1} r_{n-1}] \leq [f(x_n), k^n r_n], \; n>1.$
\end{center}
So, we obtain 
\begin{equation*}
\begin{array}{rcl}
\| x_{n+1}-x_n \|  &=& \| f(x_n)-f(x_{n-1}) \| \\
&\leq & k \| x_{n-1}-x_{n-2}\| \\
&\leq & k^2 \| x_{n-2}-x_{n-3}\| \\
&\vdots & \\
&\leq & k^{n-1} \| x_1-x_0\| \\
&\leq & k^{n-1} r_0,
\end{array}
\end{equation*}
where $r_0= \| x_1- x_0 \|$. Then, for any positive integer $p$,  one can obtain
\begin{center}
$\| x_{n+p}- x_n\| \leq k^{n+p-1} \| x_1- x_0\|.$
\end{center}
Since $k \in (0,1)$, as $n \longrightarrow \infty$, we imply that $\{ [f(x_n),k^n r_0]\}$ is a Cauchy sequence. Since  $[B_r]$ is a domain, then 
$\{ [f(x_n),k^n r_0]\}$ is convergent to its least upper bound. Now we compute its limit point. Let $[x_n , r_n] \longrightarrow [x,0].$
Since by Theorem \ref{th3.3}, $T$ is Scott--continuous, then
\begin{center}
$\vee \{T([x_n,r_n])\}= T(\vee \{[x_n,r_n] \}),$
\end{center}
and 
$$T([x,0])= \vee \{[f^n (x_0),k^n r_n] \}.$$
So
\begin{center}
$\{[f(x),0]\}= \vee \{[f^n (x_0),k^n r_n] \}.$
\end{center}
Hence from Theorem \ref{th3.3}, we obtain
\begin{center}
$\lim_{n \rightarrow \infty} [f^n (x_0),k^n r_n] = [f(x),0].$
\end{center}
But by Theorem  \ref{th2.6}, we imply that $f^n (x_0) \longrightarrow f(x)$ and $r_n \longrightarrow 0.$
This is equivalent to the fact that $[x,0]$  is a least upper bound of $\{[f^n (x_0),k^n r_n] \}.$
If we prove that the limit point is unique, this means  that $T([x,0])=[x,0]$   and $[x,0]$ is a fixed point of $T.$
So, we'll prove the uniqueness. Let $[y, \lambda] $ be another fixed point of  $T$. Then either  $[y, \lambda] \leq [x,0]$ or
$ [x,0] \leq [y, \lambda] $, since both of them are least upper bounds. Suppose that $[y, \lambda] \leq [x,0]$,
then
\[
\begin{array}{rcl}
\| y-x \| &=& \|T(y)-T(x)\| \\
&=& \|T^n (y)-T^n (x)\| \\
&\leq& k^n \|T(y)-T(x)\|.
\end{array}
\]
As $n \longrightarrow \infty,$ we have  $y=x,$ since  $k \in (0,1).$ \qed
\end{proof}
\section{Applications to Graded Formal Concepts and Closed Itemsets } \label{sec4}
When a computational method manifests itself in the form of a known mathematical structure, it has two important advantages. The first is that, the accuracy of the calculation will go up as much as possible, and the second is increasing the possibility of further development. With this approach, two important objectives are provided in this section, both are direct and indirect results of the fundamental Theorem \ref{th3.4}:

\begin{itemize}
\item[1.]
With help of this theorem, an accurate method based on constructive non-linear analysis algorithms will be presented, which calculates the formal concepts with a special method.
\item[2.]
Equipping qualitative data with the topological structure of the domain, enables us to rank the data arrangement by assigning a related measure to a set. At this point, again using the results of Theorem \ref{th3.4}, we present the most basic purpose of this paper, which is the definition of graded Galois lattices and closed itemsets.
\end{itemize}

As a general definition, a concept is an abstract idea that builds the blocks of human thought. One way to describe it, is using the notion of objects and attributes. This can be done by considering a tabular data includes extents and intents. By means of this data matrix, one can construct a complete lattice. Some special maximal elements of this lattice are called formal concepts. As another scope, we may imagine this matrix as a set of objects in the rows that have some features in the columns. In the first step, to describe it as a mathematical model, we need the definition of the context.

 \begin{definition}[\cite{Ganter1999}]\label{def4.1}
 A formal context is a triplet $(O,A,R)$ where $O$ and $A$ are non-empty sets and $R$ is a binary relation between $O$ and $A$. Elements of
 $O$ and $A$ are called objects and attributes, respectively. For $x \in O$ and $y \in A$, $(x,y) \in R$ means that the object $x$ is related to attribute $y$ by the relation $R$. Sometimes we write $xRy$. Also $(x,y) \notin R$ means that the object $x$ does not have  the  feature $y$ or $y$ is not attributed to $x$ under $R$.
 \end{definition}
 
To compute the formal concepts, there are two special operators called derivation operators.

\begin{definition}[see \cite{Ganter1999}]\label{def4.2}
For a given context $(O, A, R)$   and the sets $X \subseteq O$ and $Y \subseteq A$, we define
\begin{center}
$X' = \{y \in A: xRy, \forall x \in X\}$,\\
$Y' = \{x \in O: xRy, \forall y \in Y\}$.
\end{center}
$X'$ is called the set of attributes common to all objects in $X$. So $Y'$ is the set of objects possessing all the attributes in $Y$.
There are two related operators
\[
\begin{array}{rcl}
f: P(O) & \rightarrow & P(A) \,\, , \\
x & \mapsto & x'
\end{array}
\]
\[
\begin{array}{rcl}
g: P(A) & \rightarrow & P(O), \\
y & \mapsto & y'
\end{array}
\]
where $P(O)$ and $P(R)$ are the power sets of $O$ and $A$, respectively. The maps $f$ and $g$ are called derivation operators and in general, the pair $(f,g)$ is called the Galois connection. A formal concept in the context $(O, A, R) $ is a pair $(X, Y)$, where $X \subseteq O$, $Y \subseteq A$ and we have $f(X)=Y , g(Y)=X$. Here $X$ is called its extent and $Y$ is called its intent. In general, $(X, Y)$ may call a formal concept.
\end{definition}

In general, the Galois connection $(f,g)$ has the following properties.
\begin{itemize}
\item[(G1)] if $Y_{1}\subseteq Y_{2}$, then $g(Y_{2})\subseteq g(Y_{1})$,
\item[(G2)] $X \subseteq g(Y)$, iff $Y \subseteq f(X)$,
\item[(G3)] if $X_{1}\subseteq X_{2}$, then $f(X_{2})\subseteq f(X_{1})$,
\item[(G4)] $X\subseteq g(f(X))$,
\item[(G5)] $ Y\subseteq f(g(Y))$,
\end{itemize}
for $X, X_1, X_2 \subseteq O$ and $Y, Y_1, Y_2 \subseteq A$.\\
The conditions (G4) and (G5) lead us to set a new operator on $P(O)$, called the closure operator.

\begin{definition} [\cite{Ganter1999}] \label{def4.3}
Suppose that $(f,g)$ is a Galois connection on $(O,A,R)$. Take 
$$h=fog: P(A) \rightarrow P(A), $$
$$h'=gof: P(O) \rightarrow P(O). $$
The operators $h$ and $h'$ are called Galois closure operators.
\end{definition}

One can easily check that $h$ and $h'$ have the following properties.
\begin{itemize}
\item[(C1)] $Y \subseteq h(Y)$, (C$'$1) $X \subseteq h'(X)$  (extension),
\item[(C2)] $h(h(Y))= h(Y)$, (C$'$2) $h'(h'(X))= h'(X)$ (idempotency),
\item[(C3)] If $Y_1 \subseteq Y_2$, then $h(Y_1) \subseteq h(Y_2)$, (C$'$3) If $X_1 \subseteq X_2$, then $h'(X_1) \subseteq h'(X_2)$ (monotonicity),
\end{itemize}
for each $Y, Y_1, Y_2 \subseteq A$ and $X, X_1, X_2 \subseteq O$.\\
There is a natural way to order the concepts. This order makes the set of all concepts as a complete lattice.

\begin{definition}[\cite{Lambrechts2012}] \label{def4.4}
Assume that $(X_1,Y_1)$ and $(X_2,Y_2)$ are two concepts of the context $(O,A,R)$. We define $(X_1,Y_1) \leq (X_2,Y_2)$, if and only if
$X_1 \subseteq X_2$ (or equivalently $Y_2 \subseteq Y_1$). Here $(X_1,Y_1)$ is called a subconcept of $(X_2,Y_2)$, while $(X_2,Y_2)$
is called a superconcept of $(X_1,Y_1)$.
\end{definition}

The relation $\leq$ defines a partial order on $(O,A,R)$, called the concept order. One can show that $(G(O,A,R), \leq)$, the set of all concepts on $(O,A,R)$ with $\leq$, is a complete lattice (see \cite{Lambrechts2012}). We call it as concept lattice or Galois lattice.

There is an important notion pertained to a context, called data mining context. During the recent years, our lives have been encountered with the rapid development of information science and technology. Unlike the accumulation of big data, we have been trapped into the situation of "wealth of data and poor in knowledge". The data mining emerged, whose aim is to demonstrate and discover the potential rules and extract the useful knowledge from the data. This field of analysis has become a very active area in computer science \cite{Hualing2012}. It has a successful application in marketing analysis (\cite{Kumar2011}, \cite{Turcinek2012}), financial investment (\cite{Fiol-Roig2010}, \cite{Kovalerchuk2005}), health \cite{Jothi2015}, environmental  analysis (\cite{Pereira2008}, \cite{Tuysuzoglu2018}) and product manufacturing \cite{Vazan2017}.

\begin{definition}[Data Mining Context \cite{Pasquier1999}]\label{def4.5}
Let $\mathcal{O}$ and $\mathcal{I}$ be the finite sets of objects and items, respectively. A data mining context is a triple $\mathcal{D}=(\mathcal{O}, \mathcal{I}, R)$, where $R$ is a binary relation between the objects and items. The couple $(o,i) \in R$ means that the object $o$ is $R$-related to the item $i$.
\end{definition}

In data mining model, the transaction Ids and items will replace to the objects and attributes in the given context, respectively. So the tabular form of representing data includes a matrix form of data with transaction Ids in rows and itemsets in columns. This matrix is called a database.
In a database, closed itemsets are so important, because they are the optimal elements in the closed itemset lattice. There is a direct relation between the closed itemset lattice and Galois lattice. One can drive the closed itemset lattice from Galois lattice   \cite{Garriga2008}. Also, each closed itemset can be interpreted as the optimal collection possessing some specific features. An itemset $C \subseteq \mathcal{I}$ from a data mining $\mathcal{D}$ is closed, iff $h(C)=C$ (see Definition \ref{def4.3}).

It is so important to compute the minimal (smallest) closed itemset containing a given itemset $I \subseteq \mathcal{I}$. There are many constructive algorithms to reach the solution (for example, see \cite{Agrawal1993}, \cite{Agrawal1994}). Running time and accuracy are two essential matters. Moreover, designing a mathematical pattern to compute the solution, in general case, is so vital. We do this by computing the sequence of iterations on the closure operator. 
Let $\mathcal{C}$ be the set of closed itemsets derived from $\mathcal{D}$, and using the operator $h$. ${\mathcal{L}}_c=(\mathcal{C}, \leq)$ is a complete lattice called the closed itemset lattice \cite{Pasquier1999}.

A small finite context can easily be represented by a tabular database. But our needs in data mining like the time-series in regulatory gene expression \cite{Brazma2001} or marketing transaction \cite{Pattanayak2008} are so bigger than the normal cases. Moreover, most of computational methods such as the Next Closure Algorithm-Simulation (see \cite{Agrawal1994}) are based on closure Galois connections and closure operator.\\
Here, we are going to follow two main purposes:
\begin{itemize}
\item[(1)]
To enhance the power of data mining techniques, we see it as a set-structure mathematical object.
\item[(2)]
In large amount of data and calculation, the running time is so vital. To employ the best rate in converging to the solution, the nonlinear mathematical algorithms like iteration methods (applied in Theorem \ref{th3.4}) help us to find the optimal solution, from theoretical point of view.
\end{itemize}

A stronge algorithm to compute the concept lattice of a context is the "Next Closure" algorithm described in \cite{Belohlavek2008}. The core idea in this algorithm is based on $\downarrow\uparrow$ operator which is really a closure operator (see \cite{Belohlavek2008}). This can be taken from the Galois connections. In our paper, $\downarrow\uparrow$ is extended by a new operator that can construct without the presence of Galois connections, and from theoretical points of view, it can reduce a two sides path, into a one operation. Therefore, we expect the execution time to be significantly reduced.

At the same time, this extension can yield a good power to do the job in abstract patterns (based on extended context). Moreover, the optimal points (formal concept) are in fact the fixed points of our closure operator, which are calculated by iteration method. So, from theoretical point of view, there is really no other shorter algorithm to reach this. To see a good comparison with the Next Closure algorithm, the time complexity of the later algorithm for a context $B=(X, Y, I)$ is $O(|X|.|Y|.|B|)$. So besides of the more accuracy, in our algorithm, the time complexity is at most equal to the Next Closure one.

Before presenting our new results, we begin with the set-structure definition of a context.

\begin{definition}\label{def4.6}
Let $O$ and $A$ be two sets (called objects and attributes).  A concept operator is a monotone decreasing  map $CO: P(O) \rightarrow P(A)$.
The triplet $(O, A, CO)$ is called a context with extent and intent sets $O$ and $A$, respectively.
For $X \subseteq P(O)$ and $Y \subseteq P(A)$, we say that the pair $(X,Y)$ is a formal concept, when $CO(X)=Y$ iff $CO^{-1}(Y)=X$.
\end{definition}

Definition \ref{def4.6} helps us to apply freely the definition of concepts in the case of infinite database and even without any tabular form. According to (C3), a closure operator is always monotone increasing. This gives rise a method to find the graded Galois lattice and also the graded closed itemsets in data mining context. We remind that for any set $S$, the opposite order $\subseteq^{'}$ on $P(S)$ is defined as
\begin{center}
$A \subseteq B$, iff $B \subseteq^{'} A, \forall A, B \in P(S).$
\end{center}
It is clear that if $S$ is a complete lattice, then $(P(S), \subseteq^{'})$ is also a complete lattice. Moreover,
$$f: (P(S),\subseteq) \rightarrow (P(S),\subseteq),$$
 is ascending, iff 
$$f^{'}: (P(S), \subseteq^{'}) \rightarrow (P(S), \subseteq^{'}),$$
  is descending on the opposite poset. So the constructive result of our method can be summarized in the following diagram:
\begin{center}
\begin{tabular}{c|c}
$R$            &          $A$                     \\
\hline
$O$           &          $ \square   $          \\
\end{tabular}
$\longrightarrow$
\begin{tabular}{c|c}
$\subseteq$                      &          $(P(A), \subseteq^{'})$                     \\
\hline
$(P(O), \subseteq)$           &            $ \bigcirc    $                                  \\
\end{tabular}
\end{center}
\begin{center}
$\longrightarrow$
$\left\{
\begin{array}{l}
h: (P(O), \subseteq) \rightarrow   (P(O), \subseteq), (CO)oh: (P(O), \subseteq) \rightarrow   (P(A), \subseteq^{'}),\\
h^{c}: (P(A), \subseteq^{'}) \rightarrow   (P(A), \subseteq^{'}), (CO)^{-1}oh^{c}: (P(A), \subseteq^{'}) \rightarrow   (P(O), \subseteq).
\end{array} \right.$
\end{center}

The couple $((CO)oh, (CO)^{-1}oh^{c})$ is a concept operator and $(CO)^{-1}oh^{c}$ is also called the itemset operator. 
A pair $(X,Y) \subseteq O \times A$ is called a fixed point of $((CO)oh, (CO)^{-1}oh^{c})$, iff $(CO)^{-1}oh^{c}(Y)=Y$ and $(CO)oh(X)=X$.
According to our construction, it is clear that the formal concepts of the context $(O, A, R)$ 
are just the fixed points of $((CO)oh, (CO)^{-1}oh^{c})$ and the fixed points of $(CO)^{-1}oh^{c}$
are exactly the closed itemset of the data mining context $\mathcal{D}=(\mathcal{O}, \mathcal{I}, R)$.

All computational results in tabular data context have taken their validities from the basic concept lattice theorem (see \cite{Lambrechts2012}, Theorem 2.2.1). Also many--valued context has recently studied by researchers (for example, \cite{Qi2011}, \cite{Yan2007}).

\begin{definition} \label{def4.7}
In a many--valued context, a graded closed itemset containing the set $I \subseteq \mathcal{I}$ with grade $r$ is the least closed ball $B_r(0)$ such that 
$I \subseteq B_r(0)$. Suppose that $O \subseteq \mathcal{O}$ is the pre-image of $I$.
The least closed ball $A_s(0)$ containing $O$ is called the closed Id-transaction of grade $s$. 
$(A_s(0), B_r(0))$ is called a formal concept of grade $(r,s)$. 
$L_G=\{(A_s(0), B_r(0)): r,s \in \mathbb{N}\}$ is called the graded lattice.
\end{definition}

The following result, which is the direct consequence of Theorem \ref{th3.4}, extends the Theorem 2.2.1 in \cite{Lambrechts2012} to generalized form of a many--valued  context (see \cite{Assaghir2009}). It must be emphasized that $CO$ always induces an $h$-contraction map on any Banach lattice.

\begin{theorem} \label{th4.8}
Let $(O, A, R)$ be a many--valued context. Then $(G(O, A, R), \leq)$ is a graded complete lattice. For each $X \subseteq O$ and $Y \subseteq A$, the formal concept $(A_C, B_C)$ contains $(A, B)$ is $((CO)oh(A), (CO)^{-1}h^{c}(B))$.
\end{theorem}

Also as another result of Theorem \ref{th3.4}, one can see the following theorem.

\begin{theorem} \label{th4.9}
Suppose that $\mathcal{D}=(\mathcal{O}, \mathcal{I}, R)$ is a generalized many--valued data mining context (see Definition \ref{def4.6}) with a (finite or infinite) bounded dataset. Then the closed itemset lattice exists and for each $Y \subseteq \mathcal{I}$, the closed itemset containing
$Y$ is $ (CO)^{-1}h^{c}(Y))$.
\end{theorem}

Since both results above, have the same consequence, at the end of this section, we only consider the result of Theorem \ref{th4.9}. The importance of closed itemsets and labeled closed itemsets have been extensively discussed in \cite{Garriga2008}. But what we emphasize as the power of graded closed itemset is ordering their inclusions by level curves (here, they  are the closed balls). This will be a vital and constructive algorithm to classify and to grade the data by nonlinear classifiers. The direct consequence of these results can be applied in deep learning subject (this will be done in the next work).

Another privilege of our algorithm (in the proof of Theorem \ref{th3.4}) is to generalize the process of computation. This can be derived as a result of Theorem \ref{th3.4}. Although this result is valid for real-valued tabular databases, but it can be applied for any binary relation, since any binary   logical database is equivalent to a real-valued one.

To see a close observation, here we introduce an example. The context may be presented as a conceptual data matrix. To apply the results of 
Section \ref{sec3}, we need to change the database in many--valued context (see \cite{Assaghir2009}). In  many--valued context, the relation between objects and attributes can be presented by a parameter $\theta \in [0,1]$ as the threshold. Here we want to focus on many--valued contexts on Banach lattices.

\begin{example} \label{ex4.10}
In a health center, the following data have been recorded in Table \ref{table1} (to simplify the calculation, all data have been given between 0 and 1).
The columns $S_1, \ldots , S_4$ are symptoms of  an illness, and the rows $P_1, \ldots , P_4$ are Id transactions (list of patients). Each entry shows the degree of pain for a given patient, in terms of its symptom. Here $P= \{P_1, \ldots , P_4\}$  and $S= \{S_1, \ldots , S_4\}$.
\newpage
\begin{table}[ht]
\centering
\caption{\small A tabular data of many--valued context}
\label{table1}
\begin{tabular}{|c| c| c| c |c|}
\hline
\diagbox{$P$}{$S$}               &          $S_1$                    &              $S_2$      &            $S_3$        &             $S_4$         \\
\hline
$P_1$           &            1                          &                0.1        &            0.3             &                0            \\
\hline           
 $P_2$         &          0.3                           &               0.8        &            0.5             &                0            \\
\hline
 $P_3$         &          0.3                          &                 1          &            0.7            &                 0.5     \\
\hline
$P_4$         &          0.1                          &                 0.1          &            1             &                 1      \\
\hline
\end{tabular}
\end{table}

 Without loss of generality, we may assume that $S$ and $P$ are correspondence to the set $J_4= \{1,2,3,4\}$. Hence $S \times P$ can be represented as 16 points in the $\mathbb{R}^2$-plain. In Fig. \ref{fig1}, three of these points have been shown. 

\begin{figure}[hb]
\centering
\includegraphics[scale=0.5]{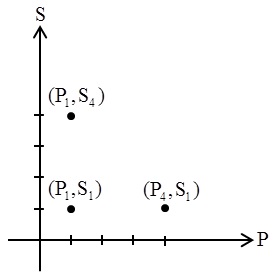}
\caption{\small Graph of $S \times P$ in $\mathbb{R}^2$}
\label{fig1}
\end{figure}

Since $\mathbb{R}^2$ is finite dimensional, all norms induce the same topology. So, we use the max--norm
$$\parallel (x,y) \parallel = \max \{|x|,|y|\}.$$

Take $k=0.1$ as the minimum of entries in the matrix of Table \ref{table1}, and define $f: \mathbb{R}^2 \rightarrow \mathbb{R}^2$ as
$f(x,y)=(kx,ky)$. A simple calculation shows that $f$ is a contraction map on $\mathbb{R}^2$. Also, one can easily show that
$CO: (\mathbb{R}^2, \parallel . \parallel) \rightarrow (\mathbb{R}^2, \parallel . \parallel)$ by $CO(A)= \bar{A}$, where $\bar{A}$ is the topological closure of $A$, is a closure operator. Now by means of Theorem \ref{th3.4}, we are able to compute the graded lattice and graded closed itemsets. Each colored rectangle in the graph of Fig. \ref{fig2} corresponds to a graded formal concept with a given grade. 

\begin{figure}[hb]
\centering
\includegraphics[scale=0.5]{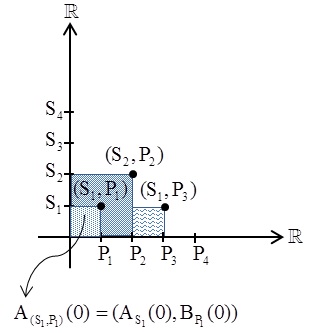}
\caption{\small Graded lattice and closed itemsets}
\label{fig2}
\end{figure}

The upper edges of each rectangle, indicates the boarder of the graded closed itemset containing its  itemsets, with the given grade. Also the black points show the nodes of the graded Galois lattice. All rectangles in Fig. \ref{fig2} can be linearized as the nested rectangles with $(0,0)$ at the origin (it is really the top of the lattice). Moreover, since $\mathbb{R}^2$ is a separable Banach lattice, one can consider a base includes the set of nested rectangles with countable widths and lengths. As a hidden observation in this example, the real power of this method is its ability to apply for continuous and uncountable datasets.
\end{example}
\section*{Conclusion} \label{concluding}
As a general overview, in this paper, a structural theorem in fixed point theory on domains has been presented and proved. The constructive results of this theorem enable us to set a new computational method based on iterations to calculate the formal concepts and closed itemsets. Also as a notable result, this leads us to present a new class of Galois lattices and closed itemsets, called graded Galois lattices and closed itemsets.

\end{document}